\documentclass[11pt]{amsart}  
\usepackage{geometry}                
\usepackage{amscd}
\usepackage{graphicx}
\usepackage{amssymb}
\usepackage{epstopdf}
\newtheorem{theorem}{Theorem}

\newtheorem{definition}[theorem]{Definition}

\newtheorem{lemma}[theorem]{Lemma}

\begin{document}

\title{ 
$K_{i}^{loc}(\mathbb{C})$, $i = 0, 1$.
}
\author{Nicolae Teleman \\
Dipartimento di Scienze Matematiche, Universita' Politecnica delle Marche \\
E-mail:  teleman@dipmat.univpm.it}
\date{}                          

\maketitle
\section{abstract}  
In this article we compute the {\em local algebraic $K$-theory}, $ i = 0, 1$, of the algebra of complex numbers $\mathbb{C}$ endowed with the  trivial filtration, i.e.  $\mathbb{C}_{\mu}= \mathbb{C}$, for any  $\mu \in \mathbb{N}$; {\em local algebras} and {\em local} algebraic $K^{loc}_{i}$-theory were introduced in \cite{Teleman_arXiv_IV}.
\par
Theorem 3 states the result. 
\par
This case corresponds in the simplest case to the Alexander-Spanier {\em local $K$-theory} over the point.
\par
This article is part of a comprehensive program aimed at re-stating the index theorem, see \cite{Teleman_arXiv_III}.
Other articles in this series are
\cite{Teleman_arXiv_IV},   \cite{Teleman_arXiv_I},  \cite{Teleman_arXiv_II}.
\section{introduction}  
\par
We recall the definition of the first two {\em local algebraic}  $K$-groups introduced in 
\cite{Teleman_arXiv_III} .  
A {\em localised algebra} $\mathit{A}$ is a filtrated associative algebra with unit; the filtration is given by a system of the vector subspaces $\mathit{A}_{\mu}$, $\mu \in \mathbb{N}$,  satisfying certain conditions, see \cite{Teleman_arXiv_IV} for precise definition.
\par
By definition {\em the group}  $K^{loc}_{0}(\mathit{A})$ {\em  of the quotient space of the space of the Grothendieck completion of the space of idempotent matrices through three equivalence relations: -i) stabilisation $\sim_{s}$, -2) {\em local} conjugation} $\sim_{l}$,  {\em and} -3) {\em projective limit with respect to the filtration}.  
\par
By definition, $K_{1}^{loc} (\mathit{A})$ {\em is the projective limit of the}  {\em local} $K_{1}(\mathit{A}_{\mu})$ {\em groups. The group}
$K_{1}(\mathit{A}_{\mu})$ {\em is by definition the quotient of } $\mathbb{GL}(\mathit{A}_{\mu})$ {\em modulo the equivalence relation
generated by: -1) stabilisation $\sim_{s}$,  --2) {\em local} conjugation $\sim_{l}$ and -3)} $\sim_{\mathbb{O}(\mathit{A}_{\mu})}$, {\em where} $\mathbb{O}(\mathit{A}_{\mu})$  {\em is the sub-module generated by elements of the form} $ u \oplus u^{-1} $,  {\em for any} $u \in \mathbb{GL}(\mathit{A}_{\mu})$.  The class of any invertible element $u$ modulo conjugation (inner auto-morphisms) is called the {\em Jordan canonical form of} $u$. 
The equivalence relation $\sim_{\mathbb{O}(\mathit{A}_{\mu})}$ insures existence of opposite elements in $K_{1}(\mathit{A}_{\mu})$ and  $K_{1}^{loc}(\mathit{A})$.
\par
\section{Notation. Preliminaries.}    
In \cite{Teleman_arXiv_IV} we  introduced the {\em local algebraic} $K_{i}^{loc}(\mathit{A})$, $i = 0, 1$, for  {\em localised algebras}.

\par
\begin{definition}  
Let $J_{n, \lambda} \in \mathbb{GL}_{n}(\mathbb{C})$, $2 \leqslant n$, denote the Jordan cell matrix of size $n$ having the
entry $\lambda$ on the diagonal
\begin{equation}  
J_{n, \lambda} =
\begin{pmatrix}
\lambda & 1 & 0 & 0 & ...... & 0 & 0 \\
0 & \lambda & 1 & 0 & ...... & 0 & 0\\
... & ... &  ... & ...  & ...  & ...  & ... \\
0 & 0 & 0 & 0 & ...... & \lambda & 1\\
0 & 0 & 0 & 0 & ...... & 0 & \lambda
\end{pmatrix}
;
\end{equation}   
it is understood that $J_{1, \lambda} = \lambda . 1_{1}$.
\par
Recall that  for of any matrix $A$,   $J(A)$ denoted the Jordan canonical form of $A$, \S 9. Remark 43.
\end{definition}    
\begin{definition}  
Let   $\bar{\mathbb{C}} := \mathbb{C} \setminus \{-1,  0,  1  \}$.
\par
Introduce on $\bar{\mathbb{C}}$ the following equivalence relation $\sim_{r}$:    $\lambda \sim_{r} \lambda^{'}$ if and only if
$\lambda = \lambda^{'}$ or $\lambda^{-1} = \lambda^{'}$.
\par
Define
\begin{equation}  
\hat{\mathbb{C}} := \bar{\mathbb{C}} /  \sim_{r}.
\end{equation}  
\end{definition}  
\section{The result and its proof.}  
\begin{theorem}   
Let $\mathbb{C}$ be endowed with the trivial filtration ($\mathbb{C}_{\mu} = \mathbb{C}$ for any $\mu \in \mathbb{N}$).
\par
Then
-i)
\begin{equation}  
 K_{0}^{loc}(\mathbb{C}) = \mathbb{Z}
 \end{equation}   
\par
-ii)
\begin{equation}  
K_{1}^{loc}(\mathbb{C}) = \\
\end{equation}   
\begin{equation*}
= \oplus_{1 \leqslant n \in \mathbb{N}} \; \; \mathbb{Z}_{2} [J_{n, -1}]  \;
\;\oplus_{2 \leqslant n \in \mathbb{N}} \; \; \mathbb{Z}_{2} [J_{n, 1}]  \oplus \\
\end{equation*}
\begin{equation*}
\oplus_{ 
1 \leqslant n \in \mathbb{N}, \;\;
\hat{\lambda} \in \hat{\mathbb{C}}} \;\;
 \mathbb{Z}  [J_{n, \hat{\lambda}}]
\end{equation*}
The group structure in $K_{1}^{loc}(\mathbb{C})$ is provided by the relations
\begin{equation*}
m_{1} [ J_{n, \lambda}]  +  m_{2} [ J_{n, \lambda^{-1}}] = ( m_{1} -  m_{2}) [ J_{n, \lambda}],    \;\;\;  for \; any \; m_{1},  m_{2} \in \mathbb{N} 
\end{equation*}
\begin{equation*}
- m [ J_{n, \lambda}]  = m [ J_{n, \lambda^{-1}}],   \;\;\;  for \; any \; m  \in \mathbb{Z}. 
\end{equation*}
\end{theorem}   
\par
The proof of this theorem gives a taste the algebraic structure behind the construction of $K_{0}^{loc}$ and $K_{1}^{loc}$.
\par
\begin{proof}  
\par
-i) It is easy to compute $K_{0}^{loc}(\mathbb{C})$. 
Any two idempotent matrices in $\mathbb{M}_{n}(\mathbb{C})$ are conjugated if and only if they have the same rank; this could be seen, for example,  by using their Jordan canonical forms. Therefore
 \begin{equation}  
 K_{0}^{loc}(\mathbb{C}) =   K_{0}(\mathbb{C}) = \mathbb{Z}.
 \end{equation}   
\par
-ii) We pass now to compute $K_{1}^{loc}(\mathbb{C})$.  We will need a few lemmas. 
\par
In the computation of $K^{loc}_{1}$,  the Jordan canonical form of the matrix  $J_{n, \lambda}^{-1}$ plays a central role.
For this reason we present two proofs of the result, formula (18)-(19).
The first proof is based on explicit formulas; the second is more conceptual.
\par
 We begin by determining  the formula for $J_{n, \lambda}^{-1}$. 
\begin{definition}  
Consider the idempotent matrix
\begin{equation}  
M_{n} :=  J_{n, \lambda} - \lambda I =
\begin{pmatrix}
0 & 1 & 0 & 0 & ... & 0 & 0 \\
0 &  0 & 1 & 0 & ... & 0 & 0\\
... & ... &  ... & ...  & ...  & ...  & ... \\
0 & 0 & 0 & 0 & ... & 0 & 1\\
0 & 0 & 0 & 0 & ... & 0 & 0
\end{pmatrix}
.
\end{equation}   
\end{definition}  
Then, for any  $1 \leqslant k \leqslant n-1$
\begin{equation}  
M_{n}^{\; k} =
\begin{pmatrix}
0 & ...& 0 & 1&0 ...  & 0 \\
0 &  ... & 0 & 0 & 1... & 0\\
... & ...   ... & ...  & ...  & ...  & ... \\
0 &  0 & 0 & ... & 0 & 1\\
... & ...   ... & ...  & ...  & ...  & ... \\
0 &  0 & 0 & ... & 0 & 0\\
0 & 0 & 0 & ... & 0 & 0
\end{pmatrix}
,
\end{equation}   
i.e., in this matrix the entries of the $(k+1)^{th}$ diagonal, above the main diagonal, are equal to $1$ while all other entries are equal to $0$.   
\par
In consecutive powers of $M_{n}$ the diagonal of $1$'s moves towards the upper right end of the matrix. 
From this it follows that any distinct powers $k$,   $1 \leqslant k \leqslant n-1$,    of $M_{n}$ are linearly independent over 
$\mathbb{C}$.
For larger powers of this matrix one has
\begin{equation}  
M_{n}^{\;k} = 0,  \;\;\;   n \leqslant k.
\end{equation}   
\par
These well known formulas allow us to find $J(J_{n, \lambda}^{-1})$. Indeed,
\begin{gather}   
I = I - (-\frac{1}{\lambda}M_{n})^{n} = (I - (\frac{-1}{\lambda}M_{n})) \sum_{0 \leqslant k \leqslant n-1}  (-\frac{1}{\lambda}M_{n}) ^{\;k} = 
\end{gather}   
\begin{gather*}   
 = (\lambda J_{n, \lambda}) \sum_{0 \leqslant k \leqslant n-1}  (-\frac{1}{\lambda}M_{n}) ^{k},
\end{gather*}   
which shows

\begin{gather}    
   (\lambda \;J_{n, \lambda})^{-1} \;=\;  \sum_{0 \leqslant k \leqslant n-1}  (-\frac{1}{\lambda}M_{n}) ^{k},
\end{gather}   
or
\begin{gather}    
   J_{n, \lambda}^{-1} \;=\;  \sum_{0 \leqslant k \leqslant n-1}  \frac{(-1)^{k}}{\lambda^{k+1}}M_{n}^{\; k}.
\end{gather}   
Therefore, for any  $1 \leqslant l \leqslant n-1$

\begin{gather}    
   (J_{n, \lambda}^{-1} - \frac{1}{\lambda} I)^{l }\;=\;  \; (\sum_{1 \leqslant k \leqslant n-1}  \frac{(-1)^{k}}{\lambda^{k+1}}M_{n}^{\; k}
   \;)^{l} \;=\;    
    \frac{(-1)^{l}}{\lambda^{2l}}M_{n}^{\; l} \;+\; higher  \; powers \; of  \; M_{n}.
\end{gather}   
These show that ($2 \leq n$) 
\begin{equation}  
1 \leqslant Rank \;  (J_{n, \lambda}^{-1} - \frac{1}{\lambda} I)^{l } \;  = l-1 \neq 0, \;\; \;  1 \leqslant l \leqslant n-1
\end{equation}  
and
\begin{equation}   
Rank \;  (J_{n, \lambda}^{-1} - \frac{1}{\lambda} I)^{n} \;  = 0.
\end{equation}   
\par
All eigen-values of $J_{n, \lambda}$ are equal to $\lambda$.
The minimum power $k$ for which
\begin{equation}  
1 \leqslant Rank ( J_{n, \lambda} - \lambda .1_{n})^{k} 
\end{equation}  
\begin{equation}  
( J_{n, \lambda} - \lambda .1_{n})^{k+1} = 0
\end{equation}  
is $n-1$. For this reason, by definition,  the number $n$ will be called the {\em Jordan rank} of the matrix  $J_{n, \lambda}$, denoted
$\mathcal{JR}(J_{n, \lambda} )$. 
\par
Formulas (13), (14) show, on the other side, that 
\begin{equation}  
\mathcal{JR}(J_{n, \lambda}^{-1}) = n = \mathcal{JR}(J_{n, \lambda})
\end{equation}  
\par
Given that there is just one Jordan cell with given Jordan rank and spectrum, this relation says that
 \begin{equation}  
 J\; ( J_{n, \lambda}^{-1} ) = J_{n, \lambda^{-1}}. 
\end{equation}  
The next lemma states the same formula (18), but provides a different proof of it.
\begin{lemma}  
Let   $J_{n, \lambda} \in \mathbb{GL}(\mathbb{C})$, $1 \leq n$. Then
\begin{equation}  
J\; ( J_{n, \lambda}^{-1} ) = J_{n, \lambda^{-1}}.
\end{equation}   
\end{lemma}   
\begin{proof}   
This a second proof of (18) which does not use the  explicit formula for $J_{n, \lambda}^{-1}$ discussed above. 
As $J_{n, \lambda} \in \mathbb{GL}(\mathbb{M}(\mathbb{C}))$, $ \lambda \neq 0$.
\par
The matrix $J_{n, \lambda} $  is invertible. We get, for any $1 \leqslant k \leq n-1$
\begin{equation}  
  1 \leqslant  Rank (J_{n, \lambda} - \lambda 1_{n})^{k}  =  Rank [\; J_{n, \lambda}^{k} . ( 1_{n} - \lambda .J_{n, \lambda}^{-1})^{k}\; ]  =
  Rank ( 1_{n} - \lambda .J_{n, \lambda}^{-1})^{k} =  Rank ( \frac{1}{\lambda}1_{n} -  J_{n, \lambda}^{-1})^{k}
\end{equation}  
\begin{equation}  
  0 =  Rank (J_{n, \lambda} - \lambda 1_{n})^{n}  =  Rank [\; J_{n, \lambda}^{n} . ( 1_{n} - \lambda .J_{n, \lambda}^{-1})^{n}\; ]  =
  Rank ( 1_{n} - \lambda .J_{n, \lambda}^{-1})^{n} =  Rank ( \frac{1}{\lambda}1_{n} -  J_{n, \lambda}^{-1})^{n}
\end{equation}  
\par
The spectrum of $J_{n, \lambda}^{-1}$ consists of the sole value $\frac{1}{\lambda}$. Formulas (20) and (21) say that 
\begin{equation}  
\mathcal{JR} (J_{n, \lambda}^{-1}) =  n = \mathcal{JR} (J_{n, \lambda}).
\end{equation}  
\par
Given that there is just one Jordan cell with given Jordan rank and spectrum, this relation proves formula (19)
 \begin{equation*}  
 J\; ( J_{n, \lambda}^{-1} ) = J_{n, \lambda^{-1}}. 
\end{equation*}  
\end{proof}   
\par
\begin{lemma}   
Let $ U =  (A \oplus A^{-1}) \in \mathbb{O}_{2n}(\mathbb{C})$, any  $1 \leqslant n$.
\par
-i) The Jordan canonical form for of $U$ is 
\begin{equation}  
U \sim_{l} J(U) =
  \oplus_{\sum n_{i}=n} ( \; J_{n_{i}, \lambda_{i}} \oplus J_{ n_{i}, \; \lambda_{i}^{-1} }\; ) =
   \sum_{\sum n_{i}=n} ( \; J_{n_{i}, \lambda_{i}} +  J_{ n_{i}, \; \lambda_{i}^{-1} }\; ) 
\end{equation}  
with
\begin{equation}  
( \; J_{n_{i}, \lambda_{i}} +   J_{n_{i}\;, \;{\lambda_{i}}^{-1}}\; ) \in \mathbb{O}_{2n}(\mathbb{C}).
\end{equation} 
\par
-ii) 
\begin{equation}  
 [J_{n_{i}\;, \;{\lambda_{i}}^{-1}} ] = - [ J_{n_{i}, \lambda_{i}} ]  \in K_{1}^{loc} (\mathbb{C}).
\end{equation}  
\end{lemma}   
\begin{proof}  
\par  A consequence of Lemma 5 is that the equivalence relation  $\sim_{\mathbb{O}}$ could produce relations between the Jordan canonical forms with eigen-values $\lambda$ and  $\lambda^{-1}$. Inner auto-morphisms (and therefore the equivalence relation $\sim_{l}$)  do not mix up Jordan canonical forms with different values of  the class $\hat{\lambda}$. Therefore, the only relations between generators of $K_{1}^{loc}(\mathbb{C})$  could occur within the Jordan canonical forms with {\em eigen-values within the same class}  $\hat{\lambda}$, or {\em separately, within} the  Jordan canonical forms with eigen-values $-1$ and $+1$. These cases  will be discussed separately.

-i)
Recall that the filtration of $\mathbb{C}$ is trivial. This means any inner auto-morhism in $\mathbb{GL}(\mathbb{M}(\mathbb{C}))$ is allowed and there is no limitation on the number of products.
\par
The matrix $U$ is the direct sum of the matrices $A$ and $A^{-1}$.  
Although we would be allowed to use any inner auto-morphism of $\mathbb{C}^{n} \oplus \mathbb{C}^{n}$
to bring the matrix $U$ to its Jordan canonical form, a direct sum of such inner auto-morphisms acting separately onto the two terms will produce the result.
\par
We use an inner auto-morphisms to bring the matrix  $A$ to its Jordan canonical form
\begin{equation}  
A \sim_{l} \oplus_{\sum n_{i}=n}  \; J_{n_{i}, \lambda_{i}}.
\end{equation}   
We use this decomposition to replace both $A$ and $A^{-1}$ in the expression of $U$ and then we use Lemma 5  to get
\begin{equation}   
U = A \oplus A^{-1} \sim_{l} \oplus_{\sum n_{i}=n}  \; (\; J_{n_{i}, \lambda_{i}} \oplus   \; J_{n_{i}, \lambda_{i}}^{-1}\;)
\sim_{l} \oplus_{\sum n_{i}=n}  \; (\; J_{n_{i}, \lambda_{i}} \oplus   \; J_{n_{i}, \lambda_{i}^{-1}}\;)
\end{equation} 
in which, any term $ J_{n_{i}, \lambda_{i}} \oplus   \; J_{n_{i}, \lambda_{i}^{-1}}  \sim_{l} 
\; J_{n_{i}, \lambda_{i}} \; +   \; J_{n_{i}, \lambda_{i}}^{-1}\;
\in \mathbb{O}_{2n}(\mathbb{C})$.
\par
The uniqueness part of the Jordan canonical form theorem assures that the decomposition  (27) is precisely (except for permutation of the pairs in the direct sum) the Jordan canonical form of $U$
\begin{equation}  
J(U) = \oplus_{\sum n_{i}=n}  \; (\; J_{n_{i}, \lambda_{i}} \oplus   \; J_{n_{i}, \lambda_{i}^{-1}}\;).
\end{equation}   
\par
Therefore, any element $A \oplus A^{-1} \in \mathbb{O}_{2n}(\mathbb{C})$ 
splits-up in a direct sum of elements in $\mathbb{O}_{2n}(\mathbb{C}^{n})$; each of these terms represents the zero element in $K_{1}^{loc}(\mathbb{C})$.
\par
-ii) As  $(\; J_{n_{i}, \lambda_{i}} \oplus   \; J_{n_{i}, \lambda_{i}}^{-1}\;) \in \mathbb{O}_{2n}(\mathbb{C})$, one has
\begin{equation}  
 0 = [\; J_{n_{i}, \lambda_{i}} \oplus   \; J_{n_{i}, \lambda_{i}}^{-1}\;] = [  J_{n_{i}, \lambda_{i}}] +  [  J_{n_{i}, \lambda_{i}^{-1}}]
 \in K_{1}^{loc}(\mathbb{C}) .
\end{equation}   
Therefore
\begin{equation}  
  [\; J_{n_{i}, \lambda_{i}^{-1}}\;] = - [ \; J_{n_{i}, \lambda_{i}}]  \in K_{1}^{loc}(\mathbb{C}).
\end{equation}  
\end{proof}  
\begin{lemma}  
-i) The Jordan cells are generators of $K_{1}^{loc}(\mathbb{C})$ subject to the following relations   
\par
\begin{equation}  
[J_{1, 1}]  = 0 
\end{equation}  
\par
\begin{equation}  
2 \; [J_{n, -1}]  = 0 \;\; for \; any  \;  1 \leqslant n \in \mathbb{N} 
\end{equation}  
\par
\begin{equation}  
2 \; [J_{n, 1}]  = 0   \;\; for \; any  \;  2 \leqslant n \in \mathbb{N}
\end{equation}  
\par
\begin{equation}  
[J_{n\;, \;\lambda}]    +  [J_{n\;, \; \lambda^{-1}}] = 0    \;\; for \; any  \;  1 \leqslant n \in \mathbb{N} \;\; and  \;\; \lambda \neq -1, 1.
\end{equation}  
\par
-ii) In $K_{1}^{loc}(\mathbb{C})$ there no other relations between the classes of Jordan cells besides (31) - (34). 
\end{lemma}  
\begin{proof}  
\par
-i) The first relation (31) is a consequence of the stabilisation equivalence relation
\begin{equation}  
A \sim_{s} A + 1_{1}   \;\; implies  \;\; [A] = [A + 1_{1}] = [A] + [1_{1}]   \in K_{1}^{loc}(\mathbb{C}).
\end{equation} 
\par 
Formulas (32) and (33) are consequences of formula (30) for the only two values of $\lambda$ which satisfy the condition
$\lambda = \lambda^{-1.}$
\par
Formula (34) is formula (30) in the remaining cases.
\par
-ii) In what follows we show that between the classes of the Jordan canonical forms, in  $K_{1}^{loc}(\mathbb{C})$,  there are no other relations  besides the relations (31) - (34).
\par
Any generator $A$ of $K_{1}^{loc}(\mathbb{C})$,  i.e. invertible matrix, 
decomposes in a direct sum of invertible Jordan cells; this decomposition is unique modulo a permutation of them.
\par
The construction of $K^{loc}_{1}$, see Definitions 19 and 20 of \cite{Teleman_arXiv_IV},
say that any element of $K_{1}^{loc}(\mathbb{C})$ is represented by $[u]$, where $u \in \mathbb{GL}_{n}(\mathbb{C})$.
We start with two generic generators in $K_{1}^{loc}(\mathbb{C})$ which represent the same element in this group and we look for all relations which will deem necessary at the level of $K_{1}^{loc}(\mathbb{C})$. Suppose then
\begin{equation}  
(u_{1} + \xi_{1})  \sim_{l} (u_{2} + \xi_{2})  \;\; with  \;\; \xi_{1}, \xi_{2} \in \mathbb{O}_{2n}(\mathbb{C}).
\end{equation}  
(in this formula we have assumed that both sides of the equation were sufficiently stabilised)
We get the {\em necessary relations} between the involved Jordan canonical forms 
\begin{equation}   
[u_{1}] = [u_{2}] \in K_{1}^{loc}(\mathbb{C}) \;\;  and \;\; [\xi_{1}] = [\xi_{2}] = 0.
\end{equation}  
We will show that relations (31) - (34) of this lemma imply formula (37).
\par
Indeed, the equation (36) says that at the level of Jordan canonical forms of both sides  
\begin{equation}  
J (u_{1} + \xi_{1})  = J (u_{2} + \xi_{2}).
\end{equation}  
\par
The discussion made in the proof of  Lemma  6. -i) shows that each of the four elements of the equation (36) may be individually
brought to its canonical Jordan form.   The uniqueness part of the Jordan canonical form theorem tells that
\begin{equation}  
J (u_{1}) \cup J(\xi_{1})  = J (u_{2}) \cup J(\xi_{2}).
\end{equation}  
Both sides of this formula have to be seen as consisting of the union of the Jordan cells, each cell counted with its multiplicity;
 i.e.  the two sides of the equation (38) (39)  consist precisely of the same Jordan cells counted with their multiplicities.
\par
We recall, Definition 2,  that $\hat{\mathbb{C}} := \mathbb{C} \setminus \{ -1, 0, 1 \} / \sim_{r}$ where $\sim_{r}$ is the equivalence relation 
$ \lambda \sim_{r} \lambda^{'}$ if and only if  $\lambda = \lambda^{'}$ or  $\lambda^{-1} = \lambda^{'}$.
  Let  $\hat{\lambda}$ be the equivalence class of $\lambda$ in
$ \hat{\mathbb{C}}$.
\par
We decompose the set of Jordan cells appearing in the formula (39) in disjoint groups $J_{n_{i, \lambda}}$, according to the  size $n_{i}$ and eigen-value $\lambda \in \mathbb{C}$
\begin{gather}  
J_{n_{i}, \lambda} (u_{1}) := 
 \{ \text{ Jordan cells of $u_{1}$ of size $n_{i}$ and eigen-value  $\lambda$ } \}
\end{gather}  

\begin{gather}  
J_{n_{i}, \lambda} (\xi_{1}) := 
 \{
\text{ 
Jordan cells of $\xi_{1}$ of size $n_{i}$ and eigen-value $\lambda$ 
}
\}
\end{gather}  

\begin{gather}   
J_{n_{i}, \lambda} (u_{2}) := 
 \{
\text{ 
Jordan cells of $u_{2}$ of size $n_{i}$ and eigen-value $\lambda$ 
}
\}. 
\end{gather}  

\begin{gather}  
J_{n_{i}, \lambda} (\xi_{2}) := 
 \{
\text{ 
Jordan cells of $\xi_{2}$ of size $n_{i}$ and eigen-value  $\lambda$ 
}
\}
\end{gather}  
\par
We will analyse first the case in which the class $\hat{\lambda} \in \hat{\mathbb{C}}$ contains to {\em distinct} elements $\lambda_{0}$
and $\lambda_{0}^{-1}$.
\par
Let $a_{n_{i}, \lambda_{0}}(u_{\alpha})$, resp. $a_{n_{i}, \lambda_{0}^{-1}}(u_{\alpha})$, represent the number of times the Jordan cell  $J_{n_{i}, \lambda_{0}}$, resp. $J_{n_{i}, \lambda_{0}^{-1}}$, appears in $J_{n_{i}, \hat{\lambda}}(u_{\alpha})$, $\alpha = 1, 2$.
\par
Analogously, let
$b_{n_{i}, \lambda_{0}}(\xi_{\alpha})$, resp. $b_{n_{i}, \lambda_{0}^{-1}}(\xi_{\alpha})$, be the number of times the Jordan cell  $J_{n_{i}, \lambda_{0}}$, resp. $J_{n_{i}, \lambda_{0}^{-1}}$, appears in $J_{n_{i}, \hat{\lambda}}(\xi_{\alpha})$, $\alpha = 1, 2$. 
\par
The formula (39) gives,  for the chosen size and eigen-values
\begin{equation} 
( a_{n_{i}, \lambda_{0}}(u_{1})+ b_{n_{i}, \lambda_{0}}(\xi_{1})  )  J_{n_{i}, \lambda_{0}} =
( a_{n_{i}, \lambda_{0}}(u_{2})+ b_{n_{i}, \lambda_{0}}(\xi_{2})  )  J_{n_{i}, \lambda_{0}} 
\end{equation}  
\begin{equation}  
( a_{n_{i}, \lambda_{0}^{-1}}(u_{1})+ b_{n_{i}, \lambda_{0}^{-1}}(\xi_{1})  )  J_{n_{i}, \lambda_{0}^{-1}} =
( a_{n_{i}, \lambda_{0}^{-1}}(u_{2})+ b_{n_{i}, \lambda_{0}^{-1}}(\xi_{2})  )  J_{n_{i}, \lambda_{0}^{-1}} 
\end{equation}  
We extract the coefficients from the formulas (44) and (45); we obtain the information
\begin{equation} 
( a_{n_{i}, \lambda_{0}}(u_{1})+ b_{n_{i}, \lambda_{0}}(\xi_{1})  )  =
( a_{n_{i}, \lambda_{0}}(u_{2})+ b_{n_{i}, \lambda_{0}}(\xi_{2})  )   
\end{equation}  
\begin{equation}  
( a_{n_{i}, \lambda_{0}^{-1}}(u_{1})+ b_{n_{i}, \lambda_{0}^{-1}}(\xi_{1})  )  =
( a_{n_{i}, \lambda_{0}^{-1}}(u_{2})+ b_{n_{i}, \lambda_{0}^{-1}}(\xi_{2})  )  
\end{equation}  

\par
On the other side, formulas  (23),  (24), (25)   give us the second group of relations between the coefficients  
\begin{gather}  
b_{n_{i}, \lambda_{0}}(\xi_{1})  = b_{n_{i}, \lambda_{0}^{-1}}(\xi_{1}) \\
b_{n_{i}, \lambda_{0}^{-1}}(\xi_{2})  = b_{n_{i}, \lambda_{0}^{-1}}(\xi_{2}).
\end{gather}  
The equations (44)-(46), (45)-(47) and  (48), (49) describe the Jordan composition of the the conjugate elements
$\tilde{u}_{1} := (u_{1} + \xi_{1}), \; \tilde{u}_{2} := (u_{2} + \xi_{2})$. 
\par
  We pass to the $K_{1}^{loc}$ classes of these invertible elements appearing in the formulas (44) - (45). We get
    
\begin{equation}  
(a_{n_{i}, \lambda_{0}}(u_{1}) + b_{n_{i}, \lambda_{0}}(\xi_{1})) [J_{n_{i}, \lambda}] = 
(a_{n_{i}, \lambda_{0}}(u_{2}) + b_{n_{i}, \lambda_{0}}(\xi_{2}))  [J_{n_{i}, \lambda}] 
\end{equation}  
and
\begin{equation}  
(a_{n_{i}, \lambda_{0}^{-1}}(u_{1}) + b_{n_{i}, \lambda_{0}^{-1}}(\xi_{1})) [J_{n_{i}, \lambda_{0}^{-1}}] =
( a_{n_{i}, \lambda_{0}^{-1}}(u_{2}) + b_{n_{i}, \lambda_{0}^{-1}}(\xi_{1}))  [J_{n_{i}, \lambda_{0}^{-1}}].
\end{equation}  
\par  
 We sum up formulas (50), (51).
We use the formulas (48), (49)  to get

\begin{equation}   
a_{n_{i}, \lambda_{0}}(u_{1}) [J_{n_{i} \lambda_{0}}]  +  a_{n_{i}, \lambda_{0}^{-1}}(u_{1}) [J_{n_{i} \lambda_{0}^{-1}}]  +
b_{n_{i}, \lambda_{0}}(\xi_{1})) ( [ J_{n_{i} \lambda_{0}}] +  [J_{n_{i} \lambda_{0}^{-1}}] ) =
\end{equation}
\begin{equation*}
= a_{n_{i}, \lambda_{0}}(u_{2}) [J_{n_{i} \lambda_{0}}]  +  a_{n_{i}, \lambda_{0}^{-1}}(u_{2}) [J_{n_{i} \lambda_{0}^{-1}}]  +
b_{n_{i}, \lambda_{0}}(\xi_{2})) (  [J_{n_{i} \lambda_{0}}]  +  [J_{n_{i} \lambda_{0}^{-1}}] ) 
\end{equation*}  
\par
We take into account formula (34) again to get rid of the terms coming from $[\xi_{1}]$ and $\xi_{2}$; formula (52) becomes
\begin{equation}  
 a_{n_{i}, \lambda_{0}}(u_{1})  [J_{n_{i}, \lambda_{0}}]+
a_{n_{i}, \lambda_{0}^{-1}}(u_{1})   [J_{n_{i}, \lambda_{0}^{-1}}]= 
\end{equation}
\begin{equation*}
=  a_{n_{i}, \lambda_{0}}(u_{2})  [J_{n_{i}, \lambda}]+
a_{n_{i}, \lambda_{0}^{-1}}(u_{2})   [J_{n_{i}, \lambda_{0}^{-1}}]. 
\end{equation*}  
This relation is precisely formula (37). 
\par
The formulas (40)  to  (53) have referred to the case of Jordan canonical forms with 
$\lambda^{-1} \neq \lambda$. We consider the remaining cases in which follows.
\par
We assume again the situation described by equations (36), (38). We need to show that (37) holds in the 
special case in which the class $\hat{\lambda} \in \hat{\mathbb{C}}$ contains just one element instead of two.
\par
We use again formulas (40) - (43) and the corresponding coefficients $a_{n_{i}, \lambda}(u_{\alpha})$, 
$b_{n_{i}, \lambda}(u_{\alpha})$, $\alpha = 1, 2$. 
\par
Keeping in mind that $\lambda^{-1} = \lambda$, one has $J(J_{n_{i}, \lambda^{-1}}) =  J(J_{n_{i}, \lambda})) $
and hence
\begin{equation}  
\xi_{\alpha} = J_{n_{i}, \lambda} \oplus J_{n_{i}, \lambda^{-1}} = 2 J_{n_{i}, \lambda}
\end{equation}  
and
\begin{equation}  
[ \xi_{\alpha}] = [ J_{n_{i}, \lambda} \oplus J_{n_{i}, \lambda}] = 2  [ J_{n_{i}, \lambda} ] = 0 \in K_{1}^{loc}(\mathbb{C})
\end{equation} 
With regard to the  previous equations (44) - (47), we see
that in  this special case we do not obtain pairs of such equations, but just one
\begin{equation} 
( a_{n_{i}, \lambda}(u_{1})+ 2 b_{n_{i}, \lambda}(\xi_{1})  )  J_{n_{i}, \lambda} =
( a_{n_{i}, \lambda}(u_{2})+ 2 b_{n_{i}, \lambda}(\xi_{2})  )  J_{n_{i}, \lambda}. 
\end{equation}  
We extract the coefficients from the formula (56); the Jordan composition of the invertible matrices of the formulas
(36), (38) provides the information
\begin{equation} 
a_{n_{i}, \lambda_{0}}(u_{1})+ 2 b_{n_{i}, \lambda_{0}}(\xi_{1})    =
 a_{n_{i}, \lambda_{0}}(u_{2})+ 2 b_{n_{i}, \lambda_{0}}(\xi_{2}).   
\end{equation}  
Formula (57) gives
\begin{equation} 
a_{n_{i}, \lambda_{0}}(u_{1})  \cong
a_{n_{i}, \lambda_{0}}(u_{2})   \;\;\; mod. 2.   
\end{equation}  
We come back to the formula (56) (which expresses the Jordan composition of the four invertible elements of formula (36)) to look for necessary relations inside $K_{1}^{loc}(\mathbb{C})$
\begin{equation} 
 a_{n_{i}, \lambda}(u_{1})  [J_{n_{i}, \lambda}] =
 a_{n_{i}, \lambda}(u_{2})  [J_{n_{i}, \lambda}]. 
\end{equation}  
Relation (58) says that the formula (59) is indeed verified if the coefficients  $a_{n_{i}, \lambda}(u_{1})$,  $a_{n_{i}, \lambda}(u_{2})$
are taken modulo $2$. This is precisely what formulas (32), (33) of the lemma say.
\par
This completes te proof of Lemma 7.
\end{proof} 
\par
We come back to the proof of Theorem 3.  Lemma 7 provides most of the information stated by Theorem 3. We only need to explain how the last term of the formula (4) appears.
We already established formula (53)
\begin{equation}  
 a_{n_{i}, \lambda_{0}}(u_{1})  [J_{n_{i}, \lambda_{0}}]+
a_{n_{i}, \lambda_{0}^{-1}}(u_{1})   [J_{n_{i}, \lambda_{0}^{-1}}]= 
 a_{n_{i}, \lambda_{0}}(u_{2})  [J_{n_{i}, \lambda}]+
a_{n_{i}, \lambda_{0}^{-1}}(u_{2})   [J_{n_{i}, \lambda_{0}^{-1}}]. 
\end{equation}  
We use relation (30) to eliminate $[J_{n_{i},\lambda^{-1}}]$ from formula (60); we get
\begin{equation}  
( a_{n_{i}, \lambda_{0}}(u_{1})  -  a_{n_{i}, \lambda_{0}^{-1}}(u_{1})  )
 [J_{n_{i}, \lambda_{0}}] =
 ( a_{n_{i}, \lambda_{0}}(u_{2})  -  a_{n_{i}, \lambda_{0}^{-1}}(u_{2})  )
 [J_{n_{i}, \lambda_{0}}] \in \mathbb{Z}
 \end{equation}  
 which explains how integer coefficients appear in the formula (4).
\par
This ends the proof of Theorem  3. 
\end{proof}  


 \end{document}